\documentclass[10pt]{article}
\usepackage{amssymb}
\usepackage{graphicx}
\usepackage{xcolor} % A package to add color.
\usepackage{tensor}
\usepackage{fullpage} % Sets all margins to 1 in.  
\usepackage{amsmath}
\usepackage{amsthm}
\usepackage{verbatim}
\usepackage{hyperref}

\usepackage{enumitem}
\setlist[enumerate]{leftmargin=1.5em}
\setlist[itemize]{leftmargin=1.5em}

\usepackage{seqsplit}
\definecolor{green}{rgb}{0,0.8,0} % Redefines the color green.
\newtheorem{theorem}{Theorem}[section]

\newtheorem{lemma}[theorem]{Lemma}
\newtheorem{proposition}[theorem]{Proposition}

\theoremstyle{definition}

\theoremstyle{remark}
\newtheorem{remark}[theorem]{Remark}

\numberwithin{equation}{section}
\newcommand{\nrm}[1]{\Vert#1\Vert}

\newcommand{\nnrm}[1]{{\vert\kern-0.25ex\vert\kern-0.25ex\vert #1 
    \vert\kern-0.25ex\vert\kern-0.25ex\vert}}

\newcommand{\rd}{\partial}
\newcommand{\nb}{\nabla}

\newcommand{\alp}{\alpha}

\newcommand{\Lmb}{\Lambda}

\newcommand{\Omg}{\Omega}

\newcommand{\bbR}{\mathbb R}

\newcommand{\bbT}{\mathbb T}

\newcommand{\bbZ}{\mathbb Z}

\newcommand{\calA}{\mathcal A}

\newcommand{\calF}{\mathcal F}

\newcommand{\calM}{\mathcal M}

\newcommand{\R}{\mathbb R}

\newcommand{\T} {\mathbb T}
\newcommand{\pa}{\partial}

\newcommand{\e}{\varepsilon}
\newcommand{\lt}{\left}
\newcommand{\rt}{\right}
\newcommand{\bq}{\begin{equation}}
\newcommand{\eq}{\end{equation}}

\newcommand{\mc}{\mathcal{C}}
\newcommand{\mf}{\mathcal{F}}
\newcommand{\mk}{\mathcal{K}}
\newcommand{\U}{\mathcal{U}}
\newcommand{\lal}{\langle}
\newcommand{\ral}{\rangle}
\def\d{\,\mathrm{d}}

\begin{document}

\title{Relaxation to fractional porous medium equation \\ from Euler--Riesz system}%: Title of the article
\author{Young-Pil Choi\thanks{Department of Mathematics, Yonsei University, Seoul 03722, Republic of Korea. E-mail: ypchoi@yonsei.ac.kr} \and In-Jee Jeong\thanks{Department of Mathematics, Korea Institute for Advanced Study, Seoul 02455, Republic of Korea. E-mail: ijeong@kias.re.kr}}
\date{\today}

%\thanks{}%
%\subjclass{}%
%\keywords{}%

%\date{\today}%
%\dedicatory{}%
%\commby{}%
% ----------------------------------------------------------------

\maketitle

% ----------------------------------------------------------------

\begin{abstract}
{We perform asymptotic analysis for the Euler--Riesz system posed in either $\T^d$ or $\R^d$ in the high-force regime and establish a quantified relaxation limit result from the Euler--Riesz system to the fractional porous medium equation. We provide a unified approach for asymptotic analysis regardless of the presence of pressure, based on the modulated energy estimates, the Wasserstein distance of order $2$, and the bounded Lipschitz distance.}
\end{abstract}

%\tableofcontents

\section{Introduction}

\subsection{The systems}

In the current work, we are interested in the asymptotic analysis for the following damped Euler--Riesz system corresponding to the high-force regime:
\begin{align}\label{eq:ER}
\left\{
\begin{aligned}
&\rd_t \rho^{(\e)} + \nb \cdot (\rho^{(\e)} u^{(\e)}) = 0, \\
&\rd_t (\rho^{(\e)} u^{(\e)}) + \nabla \cdot ( \rho^{(\e)} u^{(\e)}\otimes u^{(\e)}) +  \frac1\e c_p \nb p(\rho^{(\e)}) =  - \frac1\e \rho^{(\e)} u^{(\e)} + \frac1\e c_K \rho^{(\e)} \nb \Lmb^{\alp-d}\rho^{(\e)},
\end{aligned}
\right.
\end{align}
where $\rho^{(\e)}(t,\cdot) : \Omg \rightarrow \bbR_{+} $ and $u^{(\e)}(t,\cdot) : \Omg \rightarrow \bbR^d$ denote the density and the velocity of the fluid, respectively. The pressure $p(\rho^{(\e)})$ is given by the power-law $p(\rho) = \rho^\gamma$, for some $\gamma \geq 1$. 
Here, the domain is either $\Omg = \bbR^d$ or $\bbT^d$ and we consider the range $-2<\alp-d<0$ for the fractional Laplacian operator $\Lmb^{\alp-d} = (-\Delta)^{\frac{\alp-d}{2}}$. The case $\alp-d=-2$ corresponds to Coulomb interaction, and we shall refer to the range $-2<\alp-d<0$ as Riesz interaction. Lastly, $c_P \ge 0$ and $c_K \in \R$ are coefficients representing the strength of the pressure and Riesz interaction force, respectively. 

The system \eqref{eq:ER} has been recently investigated in \cite{CJpre,Ser20}. A rigorous derivation of the system \eqref{eq:ER} from interacting particle systems by means of {\it mean-field limits} is established in \cite{Ser20} under suitable regularity assumptions on the solutions of \eqref{eq:ER}. In \cite{CJpre}, the local-in-time existence and uniqueness of classical solutions to the system \eqref{eq:ER} without the linear damping under suitable regularity assumptions on the initial data are established. It is clear that the existence theory developed in \cite{CJpre} can be directly applied to the system \eqref{eq:ER}. 

Let us briefly explain how the system \eqref{eq:ER} behaves when $\e$ vanishes. Define the free energy $\mf: L^1_+(\Omg) \to \R$ for the system \eqref{eq:ER} by
\bq\label{def_free}
\mf(\rho^{(\e)}) := c_P \int_\Omg \U(\rho^{(\e)})\,dx - \frac{c_K}2 \int_\Omg \rho^{(\e)} \Lmb^{\alp-d}\rho^{(\e)}\,dx,
\eq
where $\U:L^1(\Omg) \to \R$ is an increasing function describing the internal energy of the density given by
\[
\U(\rho) = \left\{ \begin{array}{ll}
 \rho \ln \rho & \textrm{if $\gamma=1$,}\\[2mm]
\displaystyle \frac{1}{\gamma-1} \rho^\gamma & \textrm{if $\gamma > 1$}.
 \end{array} \right.
\]
Here $L^1_+(\Omg)$ stands for the set of nonnegative $L^1(\Omg)$ functions. Then we can rewrite the momentum equations in \eqref{eq:ER} as
\bq\label{der_formal}
\e \rd_t (\rho^{(\e)} u^{(\e)}) + \e \nabla \cdot ( \rho^{(\e)} u^{(\e)}\otimes u^{(\e)}) = - \rho^{(\e)} \nabla \frac{\delta \mf (\rho^{(\e)})}{\delta \rho^{(\e)}}  -  \rho^{(\e)} u^{(\e)},
\eq
where $\frac{\delta \mf(\rho)}{\delta \rho}$ is the variational derivative of the free energy $\mf$ with respect to $\rho$, that is
\[
 \frac{\delta \mf (\rho)}{\delta \rho} = c_P\U'(\rho)  - c_K \Lmb^{\alp-d}\rho.
\]
Thus, at the formal level the left hand side of \eqref{der_formal} converges to zero as $\e \to 0$; if $\rho^{(\e)} \to \rho$ and $\rho^{(\e)} u^{(\e)} \to \rho u$ as $\e \to 0$, we deduce from \eqref{eq:ER} the continuity equation which has a {\it gradient flow} structure \cite{JKO98, O01}:
\bq\label{eq:E1}
\pa_t  \rho + \nabla \cdot ( \rho   u) = 0 \quad \mbox{with} \quad   \rho  u =  \rho \nabla \frac{\delta \mf ( \rho)}{\delta  \rho} = \rho (c_P \nabla \U'(\rho) - c_K \nabla \Lmb^{\alpha-d}\rho), 
\eq
which can also be rewritten as the {\it fractional porous medium flow} \cite{CSV13}:
\[%\bq\label{eq:E2}
\pa_t \rho + c_K\nabla \cdot (\rho \nabla \Lmb^{\alpha-d} \rho) = c_P \Delta \rho^\gamma.
\]%\eq

The main purpose of this work is to make the above formal derivation completely rigorous. More precisely, we will provide a unified approach for the quantitative error estimate between solutions to the equations \eqref{eq:ER} and \eqref{eq:E1}. The high-force limit or strong relaxation limit has been studied for the damped Euler system \cite{CG07,HPW11,JR02,LZ16,MM90},  Euler--Poisson system \cite{Cpre, LT17}, Euler system with nonlocal forces \cite{CCT19,CPW20}. In the present work, we extend the previous results \cite{Cpre, LT17} to the Riesz interaction case. 

\subsection{Methodology}

Our main strategy is based on estimates for the {\it modulated energy}, which is also often called as {\it relative entropy}. Note that the kinetic energy $\mk$ to the system \eqref{eq:ER} is given by
\[
\mk(U) := \frac{|m|^2}{2\rho} \quad \mbox{with} \quad U = \lt( \begin{array}{c} \rho \\ m = \rho u \end{array}\rt).
\]
Then the {\it modulated kinetic energy} is given by
\[
\int_\Omg \mk(U| \bar U)\,dx := \int_\Omg \mk(U) - \mk(\bar U) - D_{\bar\rho, \bar m}\mk(\bar U) (U -  \bar U)\,dx = \frac12\int_\Omg \rho | u -   \bar u|^2\,dx,  
\]
where 
\[
\bar U = \lt( \begin{array}{c} \bar \rho \\  \bar m =  \bar\rho  \bar u \end{array}\rt).
\]
We also introduce the modulated energy associated with the free energy defined in \eqref{def_free}:
\begin{align*}
\mf(\rho | \bar\rho) &:= \mf(\rho) - \mf(\bar\rho) - \int_\Omg \frac{\delta \mf ( \bar\rho)}{\delta  \bar\rho} (\rho - \bar\rho)\,dx\cr
& = c_P \int_\Omg \U(\rho) - \U(\bar\rho) - \U'(\bar\rho)(\rho - \bar\rho)\,dx - \frac{c_K}{2}\int_\Omg (\rho - \bar\rho) \Lmb^{\alpha-d}(\rho - \bar\rho)\,dx,
\end{align*}
where the first term on the right hand side is called {\it modulated internal energy} and the second one is called {\it modulated interaction energy}.

We mainly divide the proof into two cases: pressureless and repulsive case ($c_P = 0$ and $c_K <0$) and pressure and attractive case ($c_P > 0$ and $c_K > 0$). We notice from \cite{CJpre} that the system \eqref{eq:ER} seems to be ill-posed for the pressureless and attractive case, i.e. $c_K >0$ and $c_P = 0$. Let us provide some ideas of the proof.

\medskip

\noindent \textbf{Pressureless case.} In the absence of pressure, the total energy, which is the sum of the kinetic and free energies, is not strictly convex with respect to $\rho$. Thus it is not obvious to have some convergence of $\rho^{(\e)}$ towards $\rho$. Meanwhile, the modulated interaction energy has been employed in \cite{Due16,Ser20} to study the mean-field limits for Riesz-type flows. In particular, the \textit{extension representation} for the fractional Laplacian in the whole space (proposed in \cite{CS07}) is used in \cite{Due16,Ser20}. Motivated from these works, in the whole space case we estimate the modulated interaction energy to show convergence of $\rho^{(\e)}$ towards $\rho$ in some negative Sobolev space. On the other hand, in the periodic domain case, it is unclear how to apply the extension method of \cite{CS07}. To overcome this issue, using Fourier transform and commutator estimates, we obtain a similar type of estimate for the modulated interaction energy under an additional regularity assumption on the velocity fields $u$. 

In addition to convergence of $\rho^{(\e)}$ in some negative Sobolev space, we have a stronger convergence of $\rho$ by employing the Wasserstein distance of order $2$, which is defined by
\[
\d_2(\mu, \nu):= \inf_{\pi \in \Pi(\mu,\nu)}\lt(\iint_{\Omg \times \Omg} |x - y|^2 \,\pi(dx,dy) \rt)^{1/2},
\]
for $\mu,\nu \in \mathcal{P}_2(\Omg)$, where $\Pi(\mu,\nu)$ is the set of all probability measures on $\Omg \times \Omg$ with first and second marginals $\mu$ and $\nu$ and bounded $2$-moments, respectively. Here $\mathcal{P}_2(\Omg)$ is the set of probability measures in $\Omg$ with second moment bounded. Note that $\mathcal{P}_2(\Omg)$ is a complete metric space endowed with the $2$-Wasserstein distance. {We show that the $2$-Wasserstein distance between $\rho^{(\e)}$ and $\rho$ can be controlled by the associated modulated kinetic energy; see Proposition \ref{prop_wt}. Thus the quantitative error bound on the modulated kinetic energy also gives convergence in terms of the $2$-Wasserstein distance between the densities.}

In order to show the convergence of the momentum $\rho^{(\e)} u^{(\e)}$ towards $\rho u$, we use the bounded Lipschitz distance defined by
\[
\d_{BL}(\mu, \nu) := \sup_{\phi \in \calA }\lt|\int_\calA \phi(x)(\mu(dx) - \nu(dx)) \rt|,
\]
where the admissible set $\calA$ of test functions are given by
\[
\calA := \lt\{ \phi : \Omg \to \R : \|\phi\|_{L^\infty} \leq 1 \quad \mbox{and} \quad \|\phi\|_{Lip} := \sup_{x \neq y} \frac{|\phi(x) - \phi(y)|}{|x-y|} \leq 1 \rt\}.
\]
We provide that the bounded Lipschitz distance between the momenta can be bounded by the sum of the $2$-Wasserstein distance between the associated densities and the modulated kinetic energy.

\medskip

\noindent \textbf{Pressure case.} 
With pressure, the repulsive interaction case can be easily taken into account by almost the same arguments as the above, see Section \ref{ssec:rem} (v) below. In the attractive interaction case, it is observed in \cite{CPW20, LT17} that the modulated internal energy plays a crucial role in handling the modulated Coulomb or regular interaction energy. In particular, the attractive Coulomb interaction is considered in \cite{LT17} in the periodic domain. We extend it to cover both the attractive Riesz interaction and the whole space case. Presence of pressure gives convexity of the total energy, and therefore we have the strong convergence of $(\rho^{(\e)}, \rho^{(\e)} u^{(\e)})$ towards $(\rho, \rho u)$ in some $L^p$ space.

\medskip

\noindent \textbf{Notation.} 
Let us introduce a few notations and conventions  used throughout the paper.
Since the total mass is conserved in time (see Lemma \ref{lem_energy} below),  without loss of generality, we assume that $\rho^{(\e)}$ is a probability density function, i.e. $\|\rho^{(\e)}(t,\cdot)\|_{L^1} = 1$ for all $t \geq 0$ and $\e>0$. Moreover, $L^1_2(\Omg)$ represents the space of weighted integrable functions by $1 + |x|^2$ with the norm
\[
\|f\|_{L^1_2} := \int_\Omg (1+|x|^2) f(x)\,dx.
\]
The $L^2$ based Sobolev norms are defined by $\nrm{f}_{\dot{H}^s} = \nrm{\Lmb^sf}_{L^2}$ and $\nrm{f}_{H^s}= \nrm{f}_{\dot{H}^s} + \nrm{f}_{L^2}$.
Finally, we denote by $C$ a generic positive constant, independent of $\e$ and whose value can vary from a line to another.  Now we are ready to state the main result.

\subsection{Main result} 

\begin{theorem}\label{thm_deri} Let $T>0$ and $\gamma \geq 1$. Let $(\rho^{(\e)}, u^{(\e)})$ and $\rho$ be sufficiently regular solutions to the systems \eqref{eq:ER} and \eqref{eq:E1} on the time interval $[0,T]$, respectively. Suppose that 
\bq\label{as1}
\rho^{(\e)},  \rho \in L^\infty(0,T;L^1_2(\Omg)) \quad \mbox{and} \quad u \in W^{1,\infty}((0,T) \times \Omg).
\eq
In the periodic domain case, we assume in addition that $u \in L^\infty(0,T;H^s(\T^d))$ with $s>d/2+1$.
Then we have
\begin{itemize}
\item[(i)] pressureless and repulsive case ($c_P=0$ and $c_K < 0$): There exists $C > 0$ independent of $\e>0$ such that 
%Furthermore, if we assume $\rho^{(\e)},  \rho \in L^\infty(0,T;(\calP_2 \cap L^1)(\Omg))$, $ u \in L^\infty((0,T) \times \Omg)$, and $\d_2(\rho_0^{(\e)},  \rho_0) < \infty$, then 
\begin{align}\label{res2}
\begin{aligned}
&\sup_{0 \leq t \leq T}\lt(\d_2^2(\rho^{(\e)}(t), \rho(t)) +\|(\rho^{(\e)} -  \rho)(t,\cdot)\|_{\dot{H}^{-\frac{d-\alpha}{2}}}^2  \rt) + \int_0^T \d_{BL}^2((\rho^{(\e)} u^{(\e)})(t), ( \rho   u)(t))\,dt \cr
&\quad \leq C\e \int_\Omg \rho_0^{(\e)} |u_0^{(\e)} -  u_0|^2\,dx + C\d_2^2(\rho_0^{(\e)},  \rho_0) + C \int_\Omg (\rho_0^{(\e)} -  \rho_0)\Lmb^{\alpha - d}(\rho_0^{(\e)} -  \rho_0)\,dx + C\e^2.
\end{aligned}
\end{align}
In particular, if the right hand side of \eqref{res2} converges to zero as $\e \to 0$, then we have
\begin{align*}
\rho^{(\e)} &\to \rho \quad \mbox{in } L^\infty(0,T;\dot{H}^{-\frac{d-\alpha}{2}}(\Omg)) \mbox{ and weakly-$\star$ in } L^\infty(0,T;\calM(\Omg))\cr
\rho^{(\e)}u^{(\e)} &\to \rho u \quad \mbox{weakly-$\star$ in } L^2(0,T;\calM(\Omg)).
\end{align*}
Here we denote by $\calM(\Omg)$ the space of (signed) Radon measures on $\Omg$ with finite mass.

\item[(ii)] pressure and attractive case ($c_P>0$ and $c_K > 0$): In addition, we assume the following integrability conditions for $\rho$ and $\bar\rho$:
\bq\label{condi_rho}
\int_\Omg (\rho^{(\e)})^\gamma\,dx< \infty \ \mbox{ and } \ \int_\Omg  \rho^\gamma\,dx < \infty
\eq
uniformly in $\e>0$. When $\Omg = \R^d$, we furthermore assume that $u$ satisfies $u \in L^\infty(0,T;L^{\gamma/(\gamma-1)}(\Omg))$, where $\gamma \geq 1$ is chosen as $\gamma = 2d/(2d-\alpha)$ and the strength of the attractive interaction force $c_K > 0$ is small enough compared to the pressure-coefficient $c_P > 0$. Then we have
\begin{align}\label{res3}
\begin{aligned}
&\sup_{0 \leq t \leq T}\|(\rho^{(\e)} -  \rho)(t,\cdot)\|_{L^\gamma}^2 + \int_0^T \|(\rho^{(\e)} u^{(\e)} -   \rho   u)(t,\cdot)\|_{L^1}^2\,dt\cr
&\quad \leq C\e\int_\Omg \rho_0^{(\e)} |u_0^{(\e)} -  u_0|^2\,dx + C\int_\Omg \U(\rho_0^{(\e)}| \rho_0)\,dx  + C\e^2,
\end{aligned}
\end{align}
where $C>0$ is independent of $\e>0$.

Similarly as before, if the right hand side of \eqref{res3} converges to zero as $\e \to 0$, then we have
\begin{align*}
\rho^{(\e)} &\to \rho \quad \mbox{a.e. and in } L^\infty(0,T;L^\gamma(\Omg)), \cr
\rho^{(\e)}u^{(\e)} &\to \rho u \quad \mbox{a.e. and in } L^2(0,T;L^1(\Omg))
\end{align*}
as $\e \to 0$.
\end{itemize}
\end{theorem}

\subsection{Remarks}\label{ssec:rem}

We give several remarks regarding the main statement above.

\begin{enumerate}
\item[(i)] The required regularities of solutions for Theorem \ref{thm_deri} are obtained in \cite{CJpre21, CJpre} when $\gamma=1$. To be more precise, the local-in-time existence and uniqueness of classical solutions for \eqref{eq:ER} and \eqref{eq:E1} with $\gamma=1$ are obtained in these works. On the other hand, Theorem \ref{thm_deri} can be obtained by using a rather weak regularity of solutions to the system \eqref{eq:ER}, \cite[Definition 3.1]{LT17} for instance, see also \cite{Cpre}. 
\item[(ii)] The finite second moment of $\rho^{(\e)}$ can be easily obtained. In fact, it follows from the continuity equation of \eqref{eq:ER} that
\[
\frac12\frac{d}{dt}\int_\Omg |x|^2 \rho^{(\e)}\,dx = \int_\Omg x \cdot u^{(\e)} \rho^{(\e)}\,dx.
\]
Then applying Young's inequality together with Gr\"onwall's lemma gives
\[
\int_\Omg |x|^2 \rho^{(\e)}\,dx \leq e^T \int_\Omg |x|^2 \rho_0^{(\e)}\,dx + e^T \int_0^t \int_\Omg |u^{(\e)}|^2 \rho^{(\e)}\,dxd\tau.
\]
Since the right hand side can be bounded under the assumption that $\rho^{(\e)} \in L^\infty((0,T) \times \Omg)$ and $u^{(\e)} \in L^2((0,T) \times \Omg)$ (for instance), we have the desired result. It is worth noticing that uniform-in-$\e$ bound is not necessarily required here.
\item[(iii)] One can slightly relax the assumptions \eqref{as1} on solutions in the whole space case. To be more specific, instead of \eqref{as1}, under the assumption that $\pa_t u, \nabla u \in L^\infty((0,T) \times \R^d)$, the error estimate \eqref{res2} can be replaced by
\begin{align*}
&\sup_{0 \leq t \leq T}\|(\rho^{(\e)} -  \rho)(t,\cdot)\|_{\dot{H}^{-\frac{d-\alpha}{2}}}^2 + \int_0^T  \int_{\R^d} \rho^{(\e)}(t,x) |(u^{(\e)} -  u)(t,x)|^2\,dx dt \cr
&\quad \leq C\e \int_{\R^d} \rho_0^{(\e)} |u_0^{(\e)} -  u_0|^2\,dx  + C \int_{\R^d} (\rho_0^{(\e)} -  \rho_0)\Lmb^{\alpha - d}(\rho_0^{(\e)} -  \rho_0)\,dx \cr
&\qquad + C\e^3\lt(\int_{\R^d} \rho_0^{(\e)} |u_0^{(\e)}|^2\,dx + \frac1\e \int_{\R^d} \rho_0^{(\e)} \Lmb^{\alpha-d}\rho_0^{(\e)}\,dx\rt) + C\e^2,
\end{align*}
where $C>0$ is independent of $\e> 0$. Thus we also conclude
\[
\rho^{(\e)} \to \rho \quad \mbox{in } L^\infty(0,T;\dot{H}^{-\frac{d-\alpha}{2}}(\R^d))
\]
as $\e \to 0$. We refer to Remark \ref{rmk_e} for a detailed discussion. A similar argument can be applied to the pressure and attractive case, see Remark \ref{rmk_e2}.
\item[(iv)] Let us comment on the assumption on the uniform-in-$\e$ boundedenss of the internal energy \eqref{condi_rho} for the pressure and attractive case. When $\gamma = 1$, the uniform-in-$\e$ bound assumption \eqref{condi_rho} on $\rho^{(\e)}$ is obvious. For $\gamma > 1$, it follows from Lemma \ref{lem_energy} that 
\[
\frac12\int_\Omg \rho^{(\e)}|u^{(\e)}|^2\,dx + \frac1\e \mf(\rho^{(\e)}) + \frac1\e\int_0^t \int_\Omg \rho^{(\e)}|u^{(\e)}|^2\,dxd\tau = \frac12\int_\Omg \rho_0^{(\e)}|u_0^{(\e)}|^2\,dx + \frac1\e \mf(\rho_0^{(\e)}),
\]
and thus we get
\[
\frac{c_P}{\gamma-1}\int_\Omg (\rho^{(\e)})^\gamma \,dx \leq \frac\e2\int_\Omg \rho_0^{(\e)}|u_0^{(\e)}|^2\,dx +  \mf(\rho_0^{(\e)}) + \frac{c_K}2 \int_\Omg \rho^{(\e)} \Lmb^{\alp-d}\rho^{(\e)}\,dx.
\]
On the other hand, the last term on the right hand side of the above inequality can be estimated as
\[
\frac{c_K}2 \int_\Omg \rho^{(\e)} \Lmb^{\alp-d}\rho^{(\e)}\,dx \leq Cc_K\|\rho^{(\e)}\|_{L^\theta}^2
\]
where  $\theta = 2d/(2d - \alpha)$, due to Hardy--Littlewood--Sobolev inequality (see Lemma \ref{lem_HLS} below). We then use the $L^p$ interpolation inequality to estimate
\[
\|\rho^{(\e)}\|_{L^\theta}^2 \leq \|\rho^{(\e)}\|_{L^1}^{2(1-\beta)}\|\rho^{(\e)}\|_{L^\gamma}^{2\beta} = \|\rho^{(\e)}\|_{L^\gamma}^{2\beta},
\]
where $\beta = \gamma \alpha /(2d(\gamma-1))$. Note that if $1 + \alpha/d \leq \gamma$, then $2\beta \leq \gamma$, and thus
\[
\frac{c_P}{\gamma-1}\int_\Omg (\rho^{(\e)})^\gamma \,dx \leq \frac\e2\int_\Omg \rho_0^{(\e)}|u_0^{(\e)}|^2\,dx +  \mf(\rho_0^{(\e)}) + C c_K (1 - \delta_{2\beta,\gamma}) + Cc_K\int_\Omg (\rho^{(\e)})^\gamma \,dx.
\]
In summary, if $\gamma \geq 1+ \alpha/d$ and $c_P$ is sufficiently large compared to $c_K$, then we have
\[
\int_\Omg (\rho^{(\e)})^\gamma \,dx \leq \frac\e2\int_\Omg \rho_0^{(\e)}|u_0^{(\e)}|^2\,dx +  \mf(\rho_0^{(\e)}) + C c_K (1 - \delta_{2\beta,\gamma}).
\]
\item[(v)] The result of Theorem \ref{thm_deri} (i) can be naturally extended to the pressure and repulsive case without any further difficulties since the free energy $\mf(\rho_0^{(\e)})$ is always nonnegative. Indeed, if $(\rho^{(\e)})^\gamma, \rho^\gamma \in L^\infty(0,T;L^1(\Omg))$, then we have
\begin{align*}
&\sup_{0 \leq t \leq T}\|(\rho^{(\e)} - \rho)(t,\cdot)\|_{L^\gamma}^2 + \int_0^T \|(\rho^{(\e)} u^{(\e)} -  \rho u)(t,\cdot)\|_{L^1}^2\,dt \cr
&\quad \leq C\e\int_\Omg \rho_0^{(\e)} |u_0^{(\e)} -   u_0|^2\,dx + C\int_\Omg \U(\rho_0^{(\e)}|\rho_0)\,dx + C\e^2,
\end{align*}
where $C>0$ is independent of $\e> 0$. 
\end{enumerate}

{
The rest of this paper is organized as follows. In Section \ref{sec:mod_energy}, we provide the modulated kinetic, internal, and interaction energy estimates. Section \ref{sec:proof} is devoted to proving Theorem \eqref{thm_deri}. 
}

%%%%%%%%%%%%%%%%%%%%%%%%%%%%%%%%
%
%
%
% \section{Rigorous derivation from Euler--Riesz system}
%
%
%
%%%%%%%%%%%%%%%%%%%%%%%%%%%%%%%%

\section{Modulated energy estimates}\label{sec:mod_energy}
The goal of this section is to establish modulated energy estimates for the system \eqref{eq:ER}. Before we proceed, let us begin with some standard energy estimates: 
\begin{lemma}\label{lem_energy}Let $T>0$. Let $(\rho^{(\e)}, u^{(\e)})$ be a solution to the system \eqref{eq:ER} on the time interval $[0,T]$ with sufficient regularity. Then we have
\[
\frac{d}{dt}\int_\Omg \rho^{(\e)}\,dx = 0, \quad \frac{d}{dt}\int_\Omg \rho^{(\e)} u^{(\e)}\,dx = - \frac1\e \int_\Omg \rho^{(\e)} u^{(\e)}\,dx, 
\]
and
\[
\frac{d}{dt}\lt( \frac12\int_\Omg \rho^{(\e)}|u^{(\e)}|^2\,dx + \frac1\e \mf(\rho^{(\e)})\rt) +  \frac1\e \int_\Omg \rho^{(\e)} |u^{(\e)}|^2\,dx = 0.
\]
\end{lemma}
\begin{proof} The first two assertions are clear. For the third one, a direct computation gives
\[
\frac12\frac{d}{dt}\int_\Omg \rho^{(\e)}|u^{(\e)}|^2\,dx + \frac1\e \int_\Omg \rho^{(\e)} |u^{(\e)}|^2\,dx = - \frac1\e \int_\Omg \rho^{(\e)} u^{(\e)} \cdot \nabla \frac{\delta \mf (\rho^{(\e)})}{\delta \rho^{(\e)}}\,dx.
\]
We also find
\[
\frac{d}{dt}\mf(\rho^{(\e)}) = \int_\Omg \frac{\delta \mf (\rho^{(\e)})}{\delta \rho^{(\e)}} \pa_t \rho^{(\e)}\,dx = \int_\Omg \rho^{(\e)} u^{(\e)} \cdot \nabla \frac{\delta \mf (\rho^{(\e)})}{\delta \rho^{(\e)}}\,dx.
\]
Combining those two estimates concludes the desired result.
\end{proof}

The main purpose of this section is to prove the following proposition. 
\begin{proposition}\label{prop_deri} Let $T>0$. Let $(\rho^{(\e)}, u^{(\e)})$ and $\rho$ be sufficiently regular solutions to the systems \eqref{eq:ER} and \eqref{eq:E1} with on the time interval $[0,T]$, respectively. Suppose that $\pa_t u, \nabla  u \in L^\infty((0,T) \times \Omg)$, and $\rho^{(\e)} \in L^\infty(0,T;L^1(\Omg))$. Then we have
\begin{align}\label{est_mod}
\begin{aligned}
&\frac{d}{dt}\lt(\frac12\int_\Omg \rho^{(\e)} |u^{(\e)} -  u|^2\,dx + \frac1\e\int_\Omg \calF(\rho^{(\e)}|\rho)\,dx\rt) + \frac1{2\e} \int_\Omg \rho^{(\e)} | u^{(\e)} -  u|^2\,dx\cr
&\quad \leq \frac{c_P C}{\e} (\gamma-1) \int_\Omg \U(\rho^{(\e)} | \rho) \,dx + \frac{Cc_K}{\e}\int_\Omg (\rho^{(\e)} -  \rho)\Lmb^{\alpha - d}(\rho^{(\e)} -  \rho)\,dx   + C\e,
\end{aligned}
\end{align}
%\begin{align}\label{est_mod}
%\begin{aligned}
%&\frac{d}{dt}\lt(\frac12\int_\Omg \rho |u - \bar u|^2\,dx + \frac{c_P}\e\int_\Omg \U(\rho|\bar\rho)\,dx - \frac{c_K}{2\e} \int_\Omg (\rho - \bar\rho)\Lmb^{\alpha - d}(\rho - \bar\rho)\,dx \rt) + \frac1{2\e} \int_\Omg \rho | u - \bar u|^2\,dx\cr
%&\quad \leq \frac{c_P C}{\e} (\gamma-1) \int_\Omg \U(\rho |\bar\rho) \,dx + \frac{Cc_K}{\e}\int_\Omg (\rho - \bar\rho)\Lmb^{\alpha - d}(\rho - \bar\rho)\,dx    +  C\e \int_\Omg \rho |u|^2\,dx + C\e,
%\end{aligned}
%\end{align}
where $C>0$ is independent of $\e>0$.
\end{proposition}

\begin{remark} The integral version of the modulated energy estimate can be also obtained by using the weak energy inequality, see \cite[Definition 3.1]{LT17} for instance.
\end{remark}

\begin{remark} In addition, if we assume $u \in L^\infty((0,T) \times \Omg)$, then the kinetic energy term, the third term on the right hand side of \eqref{est_mod}, will not appear. See Remark \ref{rmk_e} for details.
\end{remark}

%%%%%%%%%%%%%%%%%%%%%%%%%%%%%%%%
%
%
%
% \section{Rigorous derivation from Euler--Riesz system}
%
%
%
%%%%%%%%%%%%%%%%%%%%%%%%%%%%%%%%

\subsection{Modulated internal energy}

We first estimate the modulated internal energy:
\begin{align*}
\int_\Omg \U(\rho^{(\e)} |  \rho)\,dx &= \int_\Omg \U(\rho^{(\e)}) - \U(\rho) - \U'(\rho) (\rho^{(\e)} -  \rho)\,dx \cr
& = \left\{ \begin{array}{ll}
\displaystyle \int_\Omg \rho^{(\e)} \ln \rho^{(\e)} - \rho \ln \rho + (\rho - \rho^{(\e)})(1 + \ln \rho) \,dx& \textrm{if $\gamma=1$,}\\[2mm]
\displaystyle \frac{1}{\gamma-1} \int_\Omg   (\rho^{(\e)})^\gamma - \rho^\gamma + \gamma(\rho - \rho^{(\e)}) \rho^{\gamma-1} \, dx & \textrm{if $\gamma > 1$}.
 \end{array} \right.
\end{align*}
By Taylor's theorem, we can easily have the following lemma. 
\begin{lemma}[Lower bounds on the modulated internal energy]\label{lem_U} Let $\gamma \geq 1$. For any $\rho^{(\e)}, \rho\in (0,\infty)$, we have
\[
 \U(\rho^{(\e)}| \rho) \geq \frac\gamma2 \min\{(\rho^{(\e)})^{\gamma-2}, \rho^{\gamma-2} \}|\rho^{(\e)} -  \rho|^2.
\]
\end{lemma}

\begin{lemma}[Temporal derivative of the modulated internal energy]\label{lem_ut}  Let $T>0$. Let $(\rho^{(\e)}, u^{(\e)})$ and $\rho$ be sufficiently regular solutions to the systems \eqref{eq:ER} and \eqref{eq:E1} with on the time interval $[0,T]$, respectively. Then we have
\bq\label{est_U}
\frac{d}{dt} \int_\Omg \U(\rho^{(\e)} | \rho)\,dx =  \int_\Omg \rho^{(\e)} ( u^{(\e)} -  u) \cdot \nabla (\U'(\rho^{(\e)}) - \U'( \rho) )\,dx - (\gamma-1) \int_\Omg \U(\rho^{(\e)} | \rho) \nabla \cdot  u\,dx.
\eq
\end{lemma}
\begin{proof} We consider two cases: $\gamma > 1$ and $\gamma = 1$. For $\gamma > 1$, we observe that
\[
\U(\rho) = \frac{1}{\gamma-1} \rho^\gamma, \quad \rho \U'(\rho) = \gamma \U(\rho), \quad \mbox{and} \quad \rho\U''(\rho) = (\gamma-1) \U'(\rho).
\]
Then we estimate
\begin{align}\label{est_p}
\begin{aligned}
\frac{d}{dt} \int_\Omg \U(\rho^{(\e)} |  \rho)\,dx
&= \int_\Omg (\U'(\rho^{(\e)}) - \U'( \rho)) \pa_t \rho^{(\e)} - \U''( \rho)(\pa_t  \rho)(\rho^{(\e)} -  \rho)\,dx\cr
&= \int_\Omg \nabla (\U'(\rho^{(\e)}) - \U'(\rho)) \cdot  \rho^{(\e)} u^{(\e)}\,dx + \int_\Omg \U''(\rho)(\nabla \cdot ( \rho  u))(\rho^{(\e)} -  \rho)\,dx\cr
&=\int_\Omg \rho^{(\e)} ( u^{(\e)} -  u) \cdot \nabla (\U'(\rho^{(\e)}) - \U'( \rho) )\,dx  +  \int_\Omg \rho^{(\e)}  u  \cdot\nabla (\U'(\rho^{(\e)}) - \U'( \rho))\,dx \cr
&\quad + \int_\Omg \U''( \rho)(\nabla \cdot ( \rho  u))(\rho^{(\e)} -  \rho)\,dx\cr
&=:   I + I_1 + I_2.
\end{aligned}
\end{align}
Here, 
\begin{align*}
I_1 &= -\int_\Omg (\nabla \rho^{(\e)} \cdot  u + \rho^{(\e)} \nabla \cdot  u)  (\U'(\rho^{(\e)}) - \U'(\rho))\,dx\cr
&=-\int_\Omg \lt(\U'(\rho^{(\e)}) \nabla \rho^{(\e)} - \U'(\rho) \nabla \rho + \U'(\rho) (\nabla \rho - \nabla \rho^{(\e)}) \rt) \cdot u\,dx \cr
&\quad- \int_\Omg \lt(\U'(\rho^{(\e)}) \rho^{(\e)} - \U'( \rho)  \rho + \U'( \rho)( \rho - \rho^{(\e)}) \rt) \nabla \cdot  u\,dx\cr
&=-\int_\Omg \lt( \nabla (\U(\rho^{(\e)}) - \U( \rho)) + \U'( \rho) \nabla( \rho - \rho^{(\e)}) \rt) \cdot  u\,dx\cr
&\quad  - \int_\Omg \lt( \gamma(\U(\rho^{(\e)}) - \U( \rho)) + \U'( \rho)( \rho - \rho^{(\e)}) \rt) \nabla \cdot  u\,dx
\end{align*}
and
\begin{align*}
I_2 &= \int_\Omg \U''( \rho)(\nabla  \rho  \cdot  u +  \rho \nabla \cdot  u)(\rho^{(\e)} -  \rho)\,dx\cr
&=-\int_\Omg (\nabla \U'( \rho) )( \rho - \rho^{(\e)}) \cdot  u\,dx - (\gamma-1)\int_\Omg \U'( \rho)( \rho - \rho^{(\e)}) \nabla \cdot u\,dx.
\end{align*}
This gives
\begin{align*}
I_1 + I_2 &= - \int_\Omg \nabla (\U(\rho^{(\e)} |  \rho))\cdot  u\,dx - \gamma \int_\Omg \U(\rho^{(\e)} | \rho) \nabla \cdot  u\,dx = (1-\gamma) \int_\Omg \U(\rho^{(\e)} | \rho) \nabla \cdot   u\,dx.
\end{align*}
Combining this with \eqref{est_p} asserts \eqref{est_U} for $\gamma > 1$. In case $\gamma = 1$, we easily find
\[
I_1 = \int_\Omg  u \cdot \rho^{(\e)} \nabla (\ln \rho^{(\e)} - \ln  \rho)\,dx = \int_\Omg  u \cdot (\nabla \rho^{(\e)} - \frac{\rho^{(\e)}}{ \rho}\nabla \rho)\,dx
\]
and
\[
I_2 = \int_\Omg \nabla\cdot ( \rho  u) (\frac{\rho^{(\e)}}{ \rho} - 1)\,dx = - \int_\Omg  \rho  u \cdot (\frac{(\nabla \rho^{(\e)}) \rho - \rho^{(\e)} \nabla  \rho}{ \rho^2})\,dx =- I_1.
\]
Thus we also have the estimate \eqref{est_U} when $\gamma = 1$. 
\end{proof}

%%%%%%%%%%%%%%%%%%%%%%%%%%%%%%%%
%
%
%
% \section{Rigorous derivation from Euler--Riesz system}
%
%
%
%%%%%%%%%%%%%%%%%%%%%%%%%%%%%%%%

\subsection{Modulated interaction energy estimate} 

In this part, we discuss the temporal derivative of the modulated interaction energy. More specifically, we provide the following lemma.

\begin{lemma}[Temporal derivative of the modulated interaction energy]\label{lem_it} Let $T>0$. Let $(\rho^{(\e)}, u^{(\e)})$ and $\rho$ be sufficiently regular solutions to the systems \eqref{eq:ER} and \eqref{eq:E1} with  on the time interval $[0,T]$, respectively. Then we have
\begin{align*}
\frac{c_K}2\frac{d}{dt} \int_\Omg (\rho^{(\e)} -  \rho)\Lmb^{\alpha - d}(\rho^{(\e)} -  \rho)\,dx &\leq   c_K\int_\Omg \rho^{(\e)}( u^{(\e)} -  u) \cdot \nabla \Lmb^{\alpha - d}(\rho^{(\e)} - \rho)\,dx\cr
&\quad  + C \int_\Omg (\rho^{(\e)} -  \rho)\Lmb^{\alpha - d}(\rho^{(\e)} -  \rho)\,dx
\end{align*}
for some $C>0$ independent of $\e > 0$.
\end{lemma}

We present the details of the proof of the above lemma by dividing into two cases: $\Omg = \T^d$ or $\Omg = \R^d$. 

\subsubsection{Whole space domain case} We first notice that the Riesz interaction can be rewritten as 
\bq\label{rel_K}
\Lmb^{\alp-d}\rho = K \star \rho,
\eq
where the kernel $K$ is given by
\[
K(x) = \frac{c_{\alpha,d}}{|x|^\alpha} 
\]
for some constant $c_{\alpha,d} > 0$. We then extend it to $\R^d \times \R$ via
\[
\int_{\R^d} K((x,\xi) - (y,0)) \rho(y)\,dy =: (K \star (\rho \otimes \delta_0))(x,\xi),
\]
where we denote
\[
K(x,\xi) := \frac{c_{\alpha,d}}{|(x,\xi)|^\alpha}
\]
see \cite{CS07} for the detailed discussion on the extension problems for the fractional Laplacian. We also refer to \cite{PS17} for the periodic domain case. Then we find that the extended interaction force satisfy
\bq\label{ext_lap}
-\nabla_{(x,\xi)} \cdot \lt(|\xi|^\zeta \nabla_{(x,\xi)} K \star (\rho \otimes \delta_0) \rt) =   \rho(x) \otimes \delta_0(\xi) \quad \mbox{on} \quad \R^d \times \R
\eq
with $\zeta := \alpha + 1 - d \in (-1,1)$ in the sense of distributions.

In the following lemma, motivated from \cite{Due16, PS17, Ser20}, we show that the modulated interaction energy can be expressed in terms of the kernel $K$. 
\begin{lemma}\label{lem_modp} The modulated potential energy can be rewritten as
\[
\int_{\R^d} (\rho^{(\e)} -  \rho)\Lmb^{\alpha - d}(\rho^{(\e)} -  \rho)\,dx = \iint_{\R^d \times \R}  |\xi|^\zeta |\nabla_{(x,\xi)} K \star ((\rho^{(\e)} -  \rho) \otimes \delta_0) (x,\xi)|^2\,dxd\xi.
\]
\end{lemma}
\begin{proof} By \eqref{rel_K} and \eqref{ext_lap}, we find
\begin{align*}
&\int_{\R^d}  (\rho^{(\e)} -  \rho) \Lmb^{\alpha - d}(\rho^{(\e)} -  \rho)\,dx\cr
&\quad = \int_{\R^d}  (\rho^{(\e)} -  \rho) K\star(\rho^{(\e)} -  \rho)\,dx\cr
&\quad =  \iiint_{\R^d \times \R^d \times \R} (\rho^{(\e)} -  \rho)(x)\otimes \delta_0(\xi) K((x,\xi) - (y,0)) (\rho^{(\e)} -  \rho)(y)\,dxdyd\xi\cr
&\quad =  \iint_{\R^d \times \R} (\rho^{(\e)} -  \rho)(x)\otimes \delta_0(\xi)  (K \star ((\rho^{(\e)} - \rho) \otimes \delta_0))(x,\xi)\,dxd\xi\cr
&\quad = - \iint_{\R^d \times \R}\nabla_{(x,\xi)} \cdot \lt(|\xi|^\zeta \nabla_{(x,\xi)} K \star ((\rho^{(\e)} - \rho) \otimes \delta_0)(x,\xi) \rt) (K \star ((\rho^{(\e)} - \rho) \otimes \delta_0))(x,\xi)\,dxd\xi\cr
&\quad = \iint_{\R^d\times \R}  |\xi|^\zeta |\nabla_{(x,\xi)} K \star ((\rho^{(\e)} - \rho) \otimes \delta_0) (x,\xi)|^2\,dxd\xi.
\end{align*}
\end{proof}

In the lemma below, we show the estimate of the temporal derivative of the modulated interaction energy in case $\Omg = \R^d$. 

\begin{lemma}\label{lem:K-R-d} Let $T>0$. Let $(\rho^{(\e)}, u^{(\e)})$ and $ \rho$ be sufficiently regular solutions to the systems \eqref{eq:ER} and \eqref{eq:E1}   on the time interval $[0,T]$, respectively. Then we have
\begin{align*}
\frac{c_K}2\frac{d}{dt} \int_{\R^d} (\rho^{(\e)} -  \rho)\Lmb^{\alpha - d}(\rho^{(\e)} -  \rho)\,dx &\leq   c_K\int_{\R^d} \rho^{(\e)}( u^{(\e)} -  u) \cdot \nabla \Lmb^{\alpha - d}(\rho^{(\e)} -  \rho)\,dx \cr
&\quad + C \int_{\R^d} (\rho^{(\e)} -  \rho)\Lmb^{\alpha - d}(\rho^{(\e)} -  \rho)\,dx,
\end{align*}
where $C>0$ is independent of $\e > 0$, it depends only on $\|\nabla u\|_{L^\infty}$ and $|c_K|$, .
\end{lemma}
\begin{proof} By Lemma \ref{lem_modp}, we estimate
\begin{align*}
&\frac{c_K}2\frac{d}{dt} \int_{\R^d} (K \star (\rho^{(\e)} -  \rho)) (\rho^{(\e)} -  \rho)\,dx \cr
&\quad = c_K\int_{\R^d}  (\rho^{(\e)} u^{(\e)} -  \rho  u) \cdot \nabla K \star (\rho^{(\e)} - \rho) \,dx \cr
&\quad = c_K\int_{\R^d} \rho^{(\e)} ( u^{(\e)} - u) \cdot \nabla K \star (\rho^{(\e)} -  \rho)\,dx - c_K \int_{\R^d} (\rho - \rho^{(\e)})  u \cdot \nabla K \star (\rho^{(\e)} -  \rho)\,dx \cr
&\quad =: J + J_1.
\end{align*}
On the other hand,
\begin{align*}
J_1 &=  c_K\iiint_{\R^d \times \R^d \times \R}( \rho - \rho^{(\e)})(x)\otimes \delta_0(\xi) ( u(x), 0) \cdot \nabla_{(x,\xi)} K((x,\xi)-(y,0)) (\rho^{(\e)} -  \rho)(y)\otimes \delta_0(\xi)\,dxdyd\xi\cr
&= c_K\iint_{\R^d \times \R} (\rho - \rho^{(\e)})(x)\otimes \delta_0(\xi) ( u(x), 0) \cdot(\nabla_{(x,\xi)} K \star (\rho^{(\e)} -  \rho)\otimes \delta_0)(x,\xi)\,dxd\xi\cr
&=  - c_K\iint_{\R^d \times \R} \nabla_{(x,\xi)} \cdot (|\xi|^\zeta \nabla_{(x,\xi)} K \star ((\rho - \rho^{(\e)}) \otimes \delta_0) (x,\xi)) \cr
&\hspace{8cm} ( u(x), 0)\cdot(\nabla_{(x,\xi)} K \star (\rho^{(\e)} -  \rho)\otimes \delta_0)(x,\xi)\,dxd\xi\cr
&=  c_K\iint_{\R^d \times \R}  |\xi|^\zeta \nabla_{(x,\xi)} K \star ((  \rho - \rho^{(\e)}) \otimes \delta_0) (x,\xi) \nabla^2_{(x,\xi)} K \star (\rho^{(\e)} -  \rho)\otimes \delta_0)(x,\xi)(u(x), 0) \,dxd\xi\cr
&\quad + c_K\iint_{\R^d \times \R}  |\xi|^\zeta  \nabla_{(x,\xi)} K \star ((  \rho - \rho^{(\e)}) \otimes \delta_0) (x,\xi) \otimes \nabla_{(x,\xi)} K \star ((\rho^{(\e)} -  \rho) \otimes \delta_0) (x,\xi)\cr
&\hspace{9cm}  : \nabla_{(x,\xi)}(u(x),0)\,dxd\xi\cr
&=: J_{11}+ J_{12},
\end{align*}
where $J_{12}$ can be easily controlled by
\[
|c_K| \|\nabla  u\|_{L^\infty} \iint_{\R^d \times \R}  |\xi|^\zeta | \nabla_{(x,\xi)} K \star ((\rho^{(\e)} -  \rho) \otimes \delta_0) (x,\xi) |^2\,dxd\xi.
\]
For $J_{11}$, by the integration by parts, we get
\begin{align*}
J_{11} &=  - c_K\iint_{\R^d \times \R}  |\xi|^\zeta \nabla_{(x,\xi)} |\nabla_{(x,\xi)} K \star ((\rho^{(\e)} -  \rho) \otimes \delta_0) (x,\xi)|^2 \cdot (u(x), 0) \,dxd\xi\cr
&=  \frac{c_K}2 \iint_{\R^d \times \R}    |\nabla_{(x,\xi)} K \star ((\rho^{(\e)} -  \rho) \otimes \delta_0) (x,\xi)|^2 \nabla_{(x,\xi)} \cdot \lt((u(x), 0)|\xi|^\zeta\rt) \,dxd\xi\cr
&\leq |c_K|\|\nabla  u\|_{L^\infty}\iint_{\R^d \times \R}  |\xi|^\zeta | \nabla_{(x,\xi)} K \star ((\rho^{(\e)} -  \rho) \otimes \delta_0) (x,\xi) |^2\,dxd\xi,
\end{align*}
where we used 
\[
\nabla_{(x,\xi)} \cdot \lt(( u(x), 0)|\xi|^\zeta\rt) = |\xi|^\zeta (\nabla \cdot  u)(x).
\]
Thus we obtain
\[
J_1  \leq 2|c_K|\|\nabla u\|_{L^\infty}\iint_{\R^d \times \R}  |\xi|^\zeta | \nabla_{(x,\xi)} K \star ((\rho^{(\e)} -  \rho) \otimes \delta_0) (x,\xi) |^2\,dxd\xi.
\]
This together with Lemma \ref{lem_modp} concludes the desired result.
\end{proof}

\subsubsection{Periodic domain case}
In this part, we take into account the periodic domain case. It is worth noticing that the method based on the extension representation for the fractional Laplacian would not be applicable to this case.

\begin{lemma}\label{lem_per} Let $T>0$. Let $(\rho^{(\e)}, u^{(\e)})$ and $\rho$ be sufficiently regular solutions to the systems \eqref{eq:ER} and \eqref{eq:E1}  on the time interval $[0,T]$, respectively. Then we have
\begin{align*}
\frac{c_K}2\frac{d}{dt} \int_{\T^d} (\rho^{(\e)} - \rho)\Lmb^{\alpha - d}(\rho^{(\e)} -  \rho)\,dx &\leq   c_K\int_{\T^d} \rho^{(\e)}( u^{(\e)} -  u) \cdot \nabla \Lmb^{\alpha - d}(\rho^{(\e)} -  \rho)\,dx \cr
&\quad + C \int_{\T^d} (\rho^{(\e)} -  \rho)\Lmb^{\alpha - d}(\rho^{(\e)} -  \rho)\,dx
\end{align*}
for some $C>0$ independent of $\e > 0$, which depends on $s$ ($s > d/2+1$),  $\|u\|_{H^s}$, and $|c_K|$.
\end{lemma}
\begin{remark} Compared to the whole space case discussed in Lemma \ref{lem:K-R-d}, we need a better regularity of solutions $u$.
\end{remark}
\begin{proof}[Proof of Lemma \ref{lem_per}]
	Proceeding as in the proof of Lemma \ref{lem:K-R-d}, it suffices to obtain the bound \begin{equation*}
	\begin{split}
	\int_{\T^d} (\rho - \rho^{(\e)})  u \cdot \nabla \Lmb^{\alp-d} (\rho^{(\e)} -  \rho)\,dx   \le C\int_{\T^d} (\rho^{(\e)} -  \rho)\Lmb^{\alp-d} (\rho^{(\e)} -  \rho) \, dx. 
	\end{split}
	\end{equation*} For the simplicity of notation, let us write $g = \rho^{(\e)} -  \rho$ and $b= -(\alp-d)/2$. Then, we compute \begin{equation*}
	\begin{split}
	\int_{\T^d} gu\cdot\nb \Lmb^{-2b}g\,dx & = -\int_{\T^d} [\Lmb^{-b} \nb \cdot (ug) - (u\cdot\nb)\Lmb^{-b}g ]\Lmb^{-b}g \,dx  - \int_{\T^d} (u\cdot\nb)\Lmb^{-b}g \, \Lmb^{-b}g\,dx  \\
	& \le C\nrm{    \Lmb^{-b} \nb \cdot (ug) - (u\cdot\nb)\Lmb^{-b}g }_{L^2}\nrm{ \Lmb^{-b}g}_{L^2} + C\nrm{\nb u}_{L^\infty} \nrm{\Lmb^{-b}g}_{L^2}^2. 
	\end{split}
	\end{equation*} To continue, we consider the Fourier series of $H :=  \Lmb^{-b} \nb \cdot (ug) - (u\cdot\nb)\Lmb^{-b}g $; for $\xi\in\bbZ^d$, \begin{equation*}
	\begin{split}
	\widehat{H}(\xi) & = i\xi|\xi|^{-b} \cdot \sum_{\eta\in\bbZ^d} \hat{g}(\eta)\hat{u}(\xi-\eta) - \sum_{\eta\in\bbZ^d} \hat{u}(\xi-\eta) \cdot i\eta|\eta|^{-b} \hat{g}(\eta) \\
	& = \sum_{\eta\in\bbZ^d} i( \xi|\xi|^{-b} - \eta|\eta|^{-b} ) \cdot \hat{u}(\xi-\eta)\hat{g}(\eta).
	\end{split}
	\end{equation*} Note that $\int_{\T^d}g\,dx = 0$ implies $\hat{g}(0) = 0$. Similarly $\hat{u}(0)=0.$ Hence in the above summation we may assume that $\eta\ne0$ and $\xi-\eta\ne0$. Moreover, when $\xi=0$,\begin{equation}\label{eq:H-zero}
	\begin{split}
	\left|\widehat{H}(0)\right| = \left|\sum_{\eta\in\bbZ^d}\eta|\eta|^{-b} \cdot \hat{u}( -\eta)\hat{g}(\eta) \right| \le C\nrm{\nb u}_{L^2}\nrm{\Lmb^{-b}g}_{L^2}
	\end{split}
	\end{equation}by H\"older's inequality. Now assuming $\xi\ne 0$, we have $|\xi|, |\eta| \gtrsim 1$ and estimate \begin{equation*}
	\begin{split}
	\left|\widehat{H}(\xi)\right| &\le C \sum_{\eta\in\bbZ^d} |\xi-\eta| (|\xi|^{-b}+|\eta|^{-b}) |\hat{u}(\xi-\eta)||\hat{g}(\eta)| \\
	& \le C \sum_{\eta\in\bbZ^d} |\xi-\eta| (|\xi-\eta|^{-b}+1) |\hat{u}(\xi-\eta)| |\eta|^{-b}|\hat{g}(\eta)| .
	\end{split}
	\end{equation*} In the second inequality, we have used that \begin{equation*}
	\begin{split}
	|\xi|^{-b} \le C|\eta|^{-b}|\xi-\eta|^b 
	\end{split}
	\end{equation*} for $\xi, \eta \in \bbZ^d$ with $\xi\ne 0, \eta\ne 0, \eta-\xi\ne 0$. Therefore, with Young's convolution inequality, \begin{equation}\label{eq:H-nonzero}
	\begin{split}
	\nrm{\widehat{H}}_{\ell^2(\xi\ne 0)} \le C \nrm{ |\xi|(|\xi|^{-b}+1)\hat{u}(\xi) }_{\ell^1} \nrm{ |\xi|^{-b}\hat{g}(\xi)}_{\ell^2}     \le C  \nrm{ u }_{H^s} \nrm{ \Lmb^{-b}g}_{L^2},
	\end{split}
	\end{equation} where $s>d/2+1$. Together with \eqref{eq:H-zero} and \eqref{eq:H-nonzero}, we obtain \begin{equation*}
	\begin{split}
	\nrm{H}_{L^2}=\nrm{    \Lmb^{-b} \nb \cdot (ug) - (u\cdot\nb)\Lmb^{-b}g }_{L^2} \le C \nrm{\Lmb^{-b}g}_{L^2}
	\end{split}
	\end{equation*} with some $C = C(\nrm{u}_{H^s})>0$. This completes the proof. 
\end{proof}

\subsection{Proof of Proposition \ref{prop_deri}} 
In this subsection, we provide the details of the proof of Proposition \ref{prop_deri}. Let us first rewrite the equation \eqref{eq:E1} as 
$$
\left\{
\begin{aligned}
&\rd_t \rho + \nb \cdot (\rho u) = 0, \\
&\rd_t (\rho  u) + \nabla \cdot (  \rho  u\otimes   u)  =  - \frac1\e   \rho  u - \frac1\e  \rho \nabla \frac{\delta \mf ( \rho)}{\delta  \rho} +  \rho  e,
\end{aligned}
\right.
$$
where $e$ is given by $ e = \pa_t  u + (u \cdot \nabla)  u$. Then straightforward computations yield
\begin{align*}
\begin{aligned}
&\frac12\frac{d}{dt}\int_\Omg \rho^{(\e)} |u^{(\e)} -  u|^2\,dx + \frac1\e \int_\Omg \rho^{(\e)} | u^{(\e)} -  u|^2\,dx \cr
&\quad = - \int_\Omg \rho^{(\e)} ( u^{(\e)} -  u) \otimes ( u^{(\e)} -  u) : \nabla  u\,dx - \frac1\e\int_\Omg \rho^{(\e)} ( u^{(\e)} -  u) \cdot  \nabla \lt( \frac{\delta \mf (\rho^{(\e)})}{\delta \rho^{(\e)}}  -  \frac{\delta \mf ( \rho)}{\delta  \rho}\rt)dx \cr
&\qquad - \int_\Omg \rho^{(\e)} ( u^{(\e)} - u) \cdot e\,dx.
\end{aligned}
\end{align*}
Here, the first term on the right hand side can be easily bounded from above by
\[
C\|\nabla u\|_{L^\infty} \int_\Omg \rho^{(\e)} |u^{(\e)} - u|^2\,dx.
\]
For the second term, we write
\begin{align*}
&-\int_\Omg \rho^{(\e)} ( u^{(\e)} -  u) \cdot  \nabla \lt( \frac{\delta \mf (\rho^{(\e)})}{\delta \rho^{(\e)}}  -  \frac{\delta \mf ( \rho)}{\delta  \rho}\rt)dx\cr
&\quad = - c_P\int_\Omg \rho^{(\e)} ( u^{(\e)} -  u) \cdot \nabla (\U'(\rho^{(\e)}) - \U'( \rho) )\,dx + c_K \int_\Omg \rho^{(\e)} ( u^{(\e)} -  u) \cdot \nabla \Lmb^{\alpha-d} (\rho^{(\e)} -  \rho)\,dx\cr
&\quad =: I +  J.
\end{align*}
On the other hand, it follows from Lemmas \ref{lem_ut} and \ref{lem_it} that 
\begin{align*}
I &= -c_P\frac{d}{dt} \int_\Omg \U(\rho^{(\e)} |  \rho)\,dx + c_P(\gamma-1) \int_\Omg \U(\rho^{(\e)} | \rho) \nabla \cdot   u\,dx\cr
& \leq -c_P\frac{d}{dt} \int_\Omg \U(\rho^{(\e)} |  \rho)\,dx + c_P(\gamma-1)\|\nabla \cdot  u\|_{L^\infty} \int_\Omg \U(\rho^{(\e)} | \rho) \,dx
\end{align*}
and
\[
J \leq \frac{c_K}2\frac{d}{dt} \int_\Omg (\rho^{(\e)} -  \rho)\Lmb^{\alpha - d}(\rho^{(\e)} -  \rho)\,dx + C \int_\Omg (\rho^{(\e)} -  \rho)\Lmb^{\alpha - d}(\rho^{(\e)} -  \rho)\,dx,
\]
where $C>0$ depends only on $\|\nabla  \cdot  u\|_{L^\infty}$ and $|c_K|$. Here we used the fact that $\U(\rho^{(\e)} | \rho)  \geq 0$. 

Finally, the third term can be estimated as 
\bq\label{est_err}
\int_\Omg \rho^{(\e)} (u^{(\e)} -  u) \cdot  e\,dx \leq  \|  e\|_{L^\infty} \int_\Omg \rho^{(\e)} | u^{(\e)} -  u|\,dx  \leq \frac{1}{4\e} \int_\Omg \rho^{(\e)} |u^{(\e)} -  u|^2\,dx + C\e,
\eq
where $C>0$ depends only on $\| e\|_{L^\infty}$ and $\|\rho^{(\e)}\|_{L^\infty(0,T;L^1)}$. Here we used the assumption $u \in W^{1,\infty}((0,T) \times \Omg)$, which implies that $e \in L^\infty((0,T) \times \Omg)$.

We now combine all of the above estimates to have
\begin{align*}
&\frac{d}{dt}\lt(\frac12\int_\Omg \rho^{(\e)} |u^{(\e)} -  u|^2\,dx + \frac{c_P}\e\int_\Omg \U(\rho^{(\e)}| \rho)\,dx - \frac{c_K}{2\e} \int_\Omg (\rho^{(\e)} -  \rho)\Lmb^{\alpha - d}(\rho^{(\e)} -  \rho)\,dx \rt)\cr
&\quad  + \frac1{2\e} \int_\Omg \rho^{(\e)} | u^{(\e)} -  u|^2\,dx\cr
&\qquad \leq \frac{c_P C}{\e} (\gamma-1) \int_\Omg \U(\rho^{(\e)} | \rho) \,dx + \frac{C}{\e}\int_\Omg (\rho^{(\e)} -  \rho)\Lmb^{\alpha - d}(\rho^{(\e)} -  \rho)\,dx  + C\e,
\end{align*}
where $C>0$ depends only on $\|u\|_{W^{1,\infty}}$. 
%We finally use Lemma \ref{lem_energy} to control the kinetic energy term on the right hand side of the above inequality. 
This completes the proof.

%\begin{remark}\label{rmk_e} If we additionally assume $ u \in L^\infty((0,T) \times \Omg)$, i.e. $  e \in L^\infty((0,T) \times \Omg)$ then the estimate \eqref{est_err} can be simplified as
%\[
%\int_\Omg \rho^{(\e)} (u^{(\e)} -  u) \cdot  e\,dx \leq  \|  e\|_{L^\infty} \int_\Omg \rho^{(\e)} | u^{(\e)} -  u|\,dx  \leq \frac{1}{4\e} \int_\Omg \rho^{(\e)} |u^{(\e)} -  u|^2\,dx + C\e,
%\]
%where $C>0$ depends only on $\| e\|_{L^\infty}$ and $\|\rho^{(\e)}\|_{L^\infty(0,T;L^1)}$.  
%\end{remark}

\begin{remark}\label{rmk_e} If we only assume $\pa_t u, \nabla u \in L^\infty((0,T) \times \Omg)$, then we modify the estimate \eqref{est_err} as
\begin{align*}
&\int_\Omg \rho^{(\e)} (u^{(\e)} -  u) \cdot (\pa_t  u + ( u - u^{(\e)} + u^{(\e)})\cdot \nabla  u)\,dx\cr
&\quad \leq \|\pa_t  u\|_{L^\infty} \int_\Omg \rho^{(\e)} | u^{(\e)} -  u|\,dx + \|\nabla  u\|_{L^\infty} \int_\Omg \rho^{(\e)} |u^{(\e)} -  u|^2\,dx\cr
&\qquad + \|\nabla  u\|_{L^\infty} \int_\Omg \rho^{(\e)} |u^{(\e)} -  u| |u^{(\e)}|\,dx\cr
&\quad \leq \frac{1}{4\e} \int_\Omg \rho^{(\e)} |u^{(\e)} -  u|^2\,dx +  C\e \int_\Omg \rho^{(\e)} |u^{(\e)}|^2\,dx  + C\e
\end{align*}
for $\e > 0$ small enough, where $C>0$ depends on $\|\rho^{(\e)}\|_{L^\infty(0,T;L^1)}$, $\|\pa_t  u\|_{L^\infty}$, and $\|\nabla   u\|_{L^\infty}$. This yields
\begin{align}\label{est_me}
\begin{aligned}
&\frac{d}{dt}\lt(\frac12\int_\Omg \rho^{(\e)} |u^{(\e)} -  u|^2\,dx + \frac{c_P}\e\int_\Omg \U(\rho^{(\e)}| \rho)\,dx - \frac{c_K}{2\e} \int_\Omg (\rho^{(\e)} -  \rho)\Lmb^{\alpha - d}(\rho^{(\e)} -  \rho)\,dx \rt)\cr
&\quad  + \frac1{2\e} \int_\Omg \rho^{(\e)} | u^{(\e)} -  u|^2\,dx\cr
&\qquad \leq \frac{c_P C}{\e} (\gamma-1) \int_\Omg \U(\rho^{(\e)} | \rho) \,dx + \frac{C}{\e}\int_\Omg (\rho^{(\e)} -  \rho)\Lmb^{\alpha - d}(\rho^{(\e)} -  \rho)\,dx +  C\e \int_\Omg \rho^{(\e)} |u^{(\e)}|^2\,dx  + C\e,
\end{aligned}
\end{align}
where $C>0$ depends only on $\|\pa_t   u\|_{L^\infty}$ and $\|\nabla  u\|_{L^\infty}$. Thus in this case, the kinetic energy $\int_\Omg \rho^{(\e)} |u^{(\e)}|^2\,dx$ appears in Proposition \ref{prop_deri}. On the other hand, it can be controlled by the total energy estimate in Lemma \ref{lem_energy} that 
\[
\int_0^t \int_\Omg \rho^{(\e)} |u^{(\e)}|^2\,dx d\tau \leq \e\lt(\int_\Omg \rho_0^{(\e)} |u_0^{(\e)}|^2\,dx + \frac1\e \int_\Omg \rho_0^{(\e)} \Lmb^{\alpha-d}\rho_0^{(\e)}\,dx\rt).
\]
We then combine this with \eqref{est_me} to conclude
\begin{align*}
&\frac{d}{dt}\lt(\frac12\int_\Omg \rho^{(\e)} |u^{(\e)} -  u|^2\,dx + \frac{c_P}\e\int_\Omg \U(\rho^{(\e)}| \rho)\,dx - \frac{c_K}{2\e} \int_\Omg (\rho^{(\e)} -  \rho)\Lmb^{\alpha - d}(\rho^{(\e)} -  \rho)\,dx \rt)\cr
&\quad  + \frac1{2\e} \int_\Omg \rho^{(\e)} | u^{(\e)} -  u|^2\,dx\cr
&\qquad \leq \frac{c_P C}{\e} (\gamma-1) \int_\Omg \U(\rho^{(\e)} | \rho) \,dx + \frac{C}{\e}\int_\Omg (\rho^{(\e)} -  \rho)\Lmb^{\alpha - d}(\rho^{(\e)} -  \rho)\,dx\cr
&\qquad \quad +  C\e^2\lt(\int_\Omg \rho_0^{(\e)} |u_0^{(\e)}|^2\,dx + \frac1\e \int_\Omg \rho_0^{(\e)} \Lmb^{\alpha-d}\rho_0^{(\e)}\,dx\rt)  + C\e,
\end{align*}
where $C>0$ is independent of $\e$.
\end{remark}

%%%%%%%%%%%%%%%%%%%%%%%%%%%%%%%%
%
%
%
% \section{Rigorous derivation from Euler--Riesz system}
%
%
%
%%%%%%%%%%%%%%%%%%%%%%%%%%%%%%%%

\section{Proof of Theorem \ref{thm_deri}}\label{sec:proof}

\subsection{Pressureless and repulsive case}

In this part, we provide the details of the proof for Theorem \ref{thm_deri} when $c_P = 0$ and $c_K <0$. For simplicity, without loss of generality, we set $c_K = -1$. In this case, it follows from \eqref{prop_deri} that 
\begin{align*}
&\frac{d}{dt}\lt(\frac12\int_\Omg \rho^{(\e)} |u^{(\e)} -  u|^2\,dx + \frac{1}{2\e} \int_\Omg (\rho^{(\e)} -  \rho)\Lmb^{\alpha - d}(\rho^{(\e)} -  \rho)\,dx \rt) + \frac1{2\e} \int_\Omg \rho^{(\e)} | u^{(\e)} -  u|^2\,dx\cr
&\quad \leq \frac{C}{\e}\int_\Omg (\rho^{(\e)} -  \rho)\Lmb^{\alpha - d}(\rho^{(\e)} -  \rho)\,dx   + C\e,
\end{align*}
where $C>0$ is independent of $\e> 0$. Now we integrate it over $[0,t]$ and use Lemma \ref{lem_energy} to have
\begin{align}\label{est_nop}
\begin{aligned}
&\frac12\int_\Omg \rho^{(\e)} |u^{(\e)} -  u|^2\,dx + \frac{1}{2\e} \int_\Omg (\rho^{(\e)} -  \rho)\Lmb^{\alpha - d}(\rho^{(\e)} -  \rho)\,dx + \frac1{2\e}\int_0^t \int_\Omg \rho^{(\e)} | u^{(\e)} -  u|^2\,dxd\tau\cr
&\quad \leq \frac12\int_\Omg \rho_0^{(\e)} |u_0^{(\e)} -  u_0|^2\,dx + \frac{1}{2\e} \int_\Omg (\rho_0^{(\e)} -  \rho_0)\Lmb^{\alpha - d}(\rho_0^{(\e)} - \rho_0)\,dx \cr
&\qquad + \frac{C}{\e}\int_0^t\int_\Omg (\rho^{(\e)} -  \rho)\Lmb^{\alpha - d}(\rho^{(\e)} -  \rho)\,dxd\tau    + C\e.
\end{aligned}
\end{align}
In particular, this implies
\begin{align*}
\int_\Omg (\rho^{(\e)} -  \rho)\Lmb^{\alpha - d}(\rho^{(\e)} -  \rho)\,dx 
&\leq \e\int_\Omg \rho_0^{(\e)} |u_0^{(\e)} -  u_0|^2\,dx +   \int_\Omg (\rho_0^{(\e)} -  \rho_0)\Lmb^{\alpha - d}(\rho_0^{(\e)} -  \rho_0)\,dx  \cr
&\quad + C\int_0^t\int_\Omg (\rho^{(\e)} -  \rho)\Lmb^{\alpha - d}(\rho^{(\e)} -  \rho)\,dxd\tau   + C\e^2,
\end{align*}
and subsequently, applying Gr\"onwall's lemma to the above, we get
\begin{align*}
\int_\Omg (\rho^{(\e)} -  \rho)\Lmb^{\alpha - d}(\rho^{(\e)} -  \rho)\,dx &\leq \e\int_\Omg \rho_0^{(\e)} |u_0^{(\e)} -  u_0|^2\,dx +   \int_\Omg (\rho_0^{(\e)} -  \rho_0)\Lmb^{\alpha - d}(\rho_0^{(\e)} -  \rho_0)\,dx + C\e^2.
\end{align*}
We then combine this with \eqref{est_nop} to yield
\begin{align}\label{est_nop2}
\begin{aligned}
&\int_\Omg \rho^{(\e)} |u^{(\e)} -  u|^2\,dx + \frac{1}{\e} \int_\Omg (\rho^{(\e)} -  \rho)\Lmb^{\alpha - d}(\rho^{(\e)} -  \rho)\,dx + \frac1{\e}\int_0^t \int_\Omg \rho^{(\e)} | u^{(\e)} -  u|^2\,dxd\tau\cr
&\quad \leq C\int_\Omg \rho_0^{(\e)} |u_0^{(\e)} -  u_0|^2\,dx + \frac{C}{\e} \int_\Omg (\rho_0^{(\e)} -\rho_0)\Lmb^{\alpha - d}(\rho_0^{(\e)} -  \rho_0)\,dx + C\e,
\end{aligned}
\end{align}
where $C>0$ is independent of $\e> 0$. On the other hand, due to the symmetry of the operator $\Lmb^{\frac{\alpha-d}{2}}$, we find
\[
\int_\Omg (\rho^{(\e)} -  \rho) \Lmb^{\alpha-d}(\rho^{(\e)} -  \rho)\,dx = \int_\Omg |\Lmb^{\frac{\alpha-d}{2}}(\rho^{(\e)} -  \rho)|^2\,dx = \|(\rho^{(\e)} -  \rho)(\tau,\cdot)\|_{\dot{H}^{-\frac{d-\alpha}{2}}}^2 .
\] 

For the quantitative error estimate between densities, we show that  the $2$-Wasserstein distance between $\rho$ and $\bar\rho$ can be controlled by the modulated kinetic energy. For this, we first recall from {\cite{AGS08},\cite[Theorem 23.9]{Vi09}} the following result on the time derivative of $2$-Wasserstein distance. 
\begin{proposition}\label{prop_wt}Let $T>0$ and $\mu, \nu \in \mc([0,T); L^1_2(\Omg))$ be solutions of the following continuity equations:
\[
\pa_t \mu + \nabla \cdot (\mu \xi) = 0 \quad \mbox{and} \quad \pa_t \nu + \nabla \cdot (\nu \eta) = 0 \quad \mbox{in the sense of distributions}
\]
for locally Lipschitz vector fields $\xi$ and $\eta$ satisfying
\[
\int_0^T\lt(\int_\Omg  |\xi|^2 \,\mu(dx) + \int_\Omg  |\eta|^2 \,\nu (dx) \rt) dt < \infty.
\]
Then for almost any $t \in (0,T)$
\begin{align*}
\frac12\frac{d}{dt} \d_2^2(\mu, \nu) &= \iint_{\Omg \times \Omg}   \lal x - y, \xi(x) - \eta(y)\ral\,\pi(dx,dy) \cr
&= \int_\Omg  \lal x - \nabla_x \varphi^*(x), \xi(x) \ral \,\mu(dx) + \int_\Omg  \lal y - \nabla_y \varphi(y), \eta(y)\ral \,\nu(dy),
\end{align*}
where $\pi \in \Pi_{o}(\mu, \nu)$, $\nabla \varphi \# \nu = \mu$, and $\nabla \varphi^* \# \mu = \nu$. Here, {$\Pi_{o}(\mu, \nu)$ stands for the set of optimal couplings between $\mu$ and $\nu$}, $\mu = \nabla \varphi \# \nu$ denotes the push-forward of $\nu$ by $\nabla \varphi$, i.e. $\mu(B) = \nu(\nabla \varphi^*(B))$ for $B \subset \Omg$, and $\varphi^*$ is the Legendre transform of $\varphi$.
\end{proposition}

\begin{lemma}\label{lem_d2} Let $T>0$. Let $(\rho^{(\e)}, u^{(\e)})$ and $  \rho$ be sufficiently regular solutions to the systems \eqref{eq:ER} and \eqref{eq:E1} on the time interval $[0,T]$, respectively. Then we have
\bq\label{est_d2}
\d_2^2(\rho^{(\e)}(t),  \rho(t)) \leq C\exp\lt(C\|\nabla  u\|_{L^\infty}\rt)\lt(\d_2^2(\rho_0^{(\e)}, \rho_0) + \int_0^t \int_\Omg \rho^{(\e)}|u^{(\e)} - u|^2\,dx d\tau\rt)
\eq
for $0 \leq t \leq T$, where $C > 0$ depends only on $T$.
\end{lemma}
\begin{proof} By Proposition \ref{prop_wt}, we find
\begin{align}\label{ineq_w2}
\begin{aligned}
\frac12\frac{d}{dt} \d_2^2(\rho^{(\e)},  \rho) &= \iint_{\Omg \times \Omg}  \lal x - y, u^{(\e)}(x) - u(y)\ral\,\pi(dx,dy) \cr
&\leq \d_2(\rho^\e, \rho)\lt(\iint_{\Omg \times \Omg}   |u^{(\e)}(x) - u(y)|^2 \,\pi(dx,dy) \rt)^{1/2}.
\end{aligned}
\end{align}
On the other hand, we get
\[
\iint_{\Omg \times \Omg}  |u^{(\e)}(x) -  u(y)|^2 \,\pi(dx,dy) \leq 2\int_\Omg |u^{(\e)}(x) -  u(x)|^2 \rho^{(\e)}(x)\,dx + 2\|\nabla  u\|_{L^\infty}^2 \d_2^2(\rho^{(\e)},   \rho),
\]
{where we used the fact $\pi$ is the optimal coupling between $\rho^{(\e)}$ and $\rho$.}
This together with \eqref{ineq_w2} yields
\[
\frac{d}{dt}\d_2(\rho^{(\e)},  \rho) \leq C \lt( \int_\Omg |u^{(\e)}(x) -  u(x)|^2 \rho^{(\e)}(x)\,dx\rt)^{1/2} + C\|\nabla  u\|_{L^\infty}\d_2(\rho^{(\e)}, \rho).
\]
Applying the Gr\"onwall's lemma to the above asserts
\[
\d_2^2(\rho^{(\e)}(t),  \rho(t)) \leq C\exp\lt(C\|\nabla  u\|_{L^\infty}\rt)\lt(\d_2^2(\rho_0^{(\e)},  \rho_0) + \int_0^t \int_\Omg \rho^{(\e)}|u^{(\e)} -   u|^2\,dx d\tau\rt),
\]
where $C > 0$ depends only on $T$.
\end{proof} 

\begin{remark}Lemma \ref{lem_d2} requires rather strong regularities of solutions to the systems \eqref{eq:ER} and \eqref{eq:E1}. To be more specific, as stated in Proposition \ref{prop_wt}, the corresponding velocity fields $u$ and $\bar u$ should be locally Lipschitz. However, this assumption can be relaxed by employing a probabilistic representation formula for continuity equations, see \cite{CCpre, CC20, Cpre, FK19} for detailed discussion. 
\end{remark}

For the quantitative bound on the second term on the left hand side of \eqref{res2}, we obtain that for any $\phi \in (L^\infty \cap Lip)(\Omg)$, 
\begin{align*}
&\int_\Omg \phi (\rho^{(\e)} u^{(\e)} -  \rho u)\,dx \cr
&\quad = \int_\Omg \phi  (\rho^{(\e)} - \rho)  u\,dx + \int_\Omg \phi \rho^{(\e)} (u^{(\e)} -  u)\,dx\cr
&\quad \leq \|\phi u\|_{L^\infty \cap Lip}\, \d_{BL}(\rho^{(\e)}, \rho) + \|\phi\|_{L^\infty} + \|\phi\|_{L^\infty} \|\rho^{(\e)}\|_{L^1}^{1/2}\lt(\int_\Omg \rho^{(\e)} |u^{(\e)} -  u|^2\,dx \rt)^{1/2}\cr
&\quad \leq C\d_2(\rho^{(\e)}, \rho) + C\lt(\int_\Omg \rho^{(\e)} |u^{(\e)} -  u|^2\,dx \rt)^{1/2}
\end{align*}
due to $\d_{BL}(\rho^{(\e)}, \rho) \leq \d_2(\rho^{(\e)}, \rho)$. This together with Lemma \ref{lem_d2} implies
\begin{align*}
\d_{BL}^2(\rho^{(\e)} u^{(\e)},  \rho   u) &\leq C\d_2^2(\rho^{(\e)}, \rho) + C\int_\Omg \rho^{(\e)} |u^{(\e)} -  u|^2\,dx \cr
&\leq C\d_2^2(\rho_0^{(\e)},  \rho_0) + C\int_0^T \int_\Omg \rho^{(\e)}|u^{(\e)} -  u|^2\,dx d\tau + C\int_\Omg \rho^{(\e)} |u^{(\e)} -   u|^2\,dx,
\end{align*}
and subsequently,
\bq\label{est_dBL}
\int_0^T \d_{BL}^2((\rho^{(\e)} u^{(\e)})(t), ( \rho  u)(t))\,dt \leq C\d_2^2(\rho_0^{(\e)},  \rho_0) + C\int_0^T \int_\Omg \rho^{(\e)}|u^{(\e)} -   u|^2\,dx dt,
\eq
where $C>0$ is independent of $\e>0$.

We finally combine \eqref{est_d2}, \eqref{est_dBL}, \eqref{est_nop2}, and Remark \ref{rmk_e} to assert
\begin{align*}
&\sup_{0 \leq t \leq T}\lt(\d_2^2(\rho^{(\e)}(t), \rho(t)) +\|(\rho^{(\e)} -  \rho)(t,\cdot)\|_{\dot{H}^{-\frac{d-\alpha}{2}}}^2  \rt) + \int_0^T \d_{BL}^2((\rho^{(\e)} u^{(\e)})(t), ( \rho  u)(t))\,dt \cr
&\quad \leq C\e \int_\Omg \rho_0^{(\e)} |u_0^{(\e)} -  u_0|^2\,dx + C\d_2^2(\rho_0^{(\e)},  \rho_0) + C \int_\Omg (\rho_0^{(\e)} -  \rho_0)\Lmb^{\alpha - d}(\rho_0^{(\e)} - \rho_0)\,dx  + C\e^2,
\end{align*}
where $C > 0$ is independent of $\e> 0$. This completes the proof.

%%%%%%%%%%%%%%%%%%%%%%%%%%%%%%%%
%
%
%
% \section{Rigorous derivation from Euler--Riesz system}
%
%
%
%%%%%%%%%%%%%%%%%%%%%%%%%%%%%%%%

\subsection{Pressure and attractive case}

We first recall Hardy--Littlewood--Sobolev inequality.
\begin{lemma}[\cite{Li83}]\label{lem_HLS} For all $f \in L^p(\Omg), g \in L^q(\Omg)$, $1 < p,q < \infty$, $d-2 < \alpha < d$ and $\frac1p + \frac1q + \frac\alpha d = 2$, it holds
\[
\lt|\int_\Omg f \Lmb^{\alpha-d}g\,dx \rt| \leq C \|f\|_{L^p}\|g\|_{L^q},
\]
where $C = C(\alpha,d,p,q) > 0$.
\end{lemma}

%
%We also recall from \cite[Lemma 3.4]{LT17} the following lemma showing the lower bound on the modulated internal energy. \Red{YP: In fact, I didn't use this lemma. It seems that this is only useful for the bounded domain case.}
%\begin{lemma}\label{lem_LT}Let $\gamma > 1$. If $\bar\rho \in I := [\delta, \bar R] \subset (0,\infty)$, then there exist positive constants $R_0 = R_0(I)$, $C_1=C_1(R_0, I)$, and $C_2 = C_2(R_0, I)$ such that 
%\begin{align*}
%\U(\rho| \bar \rho) & \geq \left\{ \begin{array}{ll}
%C_1|\rho - \bar\rho|^2& \textrm{for $0 \leq \rho \leq R_0$,}\\[3mm]
%C_2|\rho - \bar\rho|^\gamma & \textrm{for $\rho > R_0$}.
% \end{array} \right.
%\end{align*}
%\end{lemma}

\begin{lemma}\label{lem_U2}Let $\gamma\geq 1$. Suppose that $\rho^\gamma, \bar\rho^\gamma \in L^1(\Omg)$. Then we have
\[
\|\rho - \bar\rho\|_{L^\gamma}^2 \leq C\int_\Omg \U(\rho| \bar \rho)\,dx
\]
for some $C>0$ which depends only on $\|\rho\|_{L^\gamma}$, $\|\bar\rho\|_{L^\gamma}$, and $\gamma$.
\end{lemma}
\begin{proof}
We estimate
\begin{align*}
\|\rho - \bar\rho\|_{L^\gamma}^\gamma &= \int_\Omg \lt(\frac\gamma2 \min\{\rho^{\gamma-2}, \bar\rho^{\gamma-2} \} \rt)^{\frac\gamma2}\lt(\frac\gamma2 \min\{\rho^{\gamma-2}, \bar\rho^{\gamma-2} \} \rt)^{-\frac\gamma2} |\rho - \bar\rho|^\gamma\,dx\cr
&\leq \lt(\frac\gamma2\rt)^{-\frac\gamma2} \lt( \int_\Omg \frac\gamma2 \min\{\rho^{\gamma-2}, \bar\rho^{\gamma-2} \} |\rho - \bar\rho|^2\,dx\rt)^{\frac\gamma2}\lt(\int_\Omg \max\{\rho^\gamma, \bar\rho^\gamma \}\,dx \rt)^{\frac{2-\gamma}{2}}\cr
&\leq C\lt( \int_\Omg \frac\gamma2 \min\{\rho^{\gamma-2}, \bar\rho^{\gamma-2} \} |\rho - \bar\rho|^2\,dx\rt)^{\frac\gamma2},
\end{align*}
and thus
\[
\|\rho - \bar\rho\|_{L^\gamma}^2 \leq C\int_\Omg \frac\gamma2 \min\{\rho^{\gamma-2}, \bar\rho^{\gamma-2} \} |\rho - \bar\rho|^2\,dx.
\]
We then use Lemma \ref{lem_U} to conclude the desired result.
\end{proof}

We now provide the details on the proof of Theorem \ref{thm_deri} in pressure and attractive case. 

\begin{proof}[Proof of Theorem \ref{thm_deri} (ii)] By Lemma \ref{lem_HLS}, we obtain
\[
\lt| \frac{c_K}{2\e} \int_\Omg (\rho^{(\e)} -  \rho)\Lmb^{\alpha - d}(\rho^{(\e)} -  \rho)\,dx\rt| \leq \frac{Cc_K}{\e}\|\rho^{(\e)} -  \rho\|_{L^\theta}^2,
\]
where $\theta$ is given by
\[
\theta = \frac{2d}{2d-\alpha} \in \lt(\frac{2d}{d+2},2\rt).
\]
When $\Omg= \T^d$, we use the assumption $\gamma > \theta$, the monotonicity of $L^p$ norm, and Lemma \ref{lem_U2} to estimate
\[
\|\rho^{(\e)} -  \rho\|_{L^\theta}^2 \leq C\|\rho^{(\e)} -  \rho\|_{L^\gamma}^2 \leq C\int_\Omg \U(\rho^{(\e)}|   \rho)\,dx.
\]
%\begin{align}
%\begin{aligned}
%\|\rho - \bar\rho\|_{L^\theta}^2 &\leq C\lt( \int_{\T^d \cap \{ \rho \leq R_0\}} |\rho - \bar\rho|^\theta\,dx \rt)^{\frac2\theta} + C\lt( \int_{\T^d \cap \{ \rho > R_0\}} |\rho - \bar\rho|^\theta\,dx \rt)^{\frac2\theta}\cr
%&\leq C \int_{\T^d \cap \{ \rho \leq R_0\}} |\rho - \bar\rho|^2\,dx + C\lt( \int_{\T^d \cap \{ \rho > R_0\}} |\rho - \bar\rho|^\gamma\,dx \rt)^{\frac2\gamma}
%\end{aligned}
%\end{align}
In case $\Omg= \R^d$, we cannot employ the monotonicity of $L^p$ norm, thus we simply take $\gamma = \theta$ and apply Lemma \ref{lem_U2}. Hence for both cases we have
\[
\lt| \frac{c_K}{2\e} \int_\Omg (\rho^{(\e)} -  \rho)\Lmb^{\alpha - d}(\rho^{(\e)} - \rho)\,dx\rt| \leq  \frac{C_0 c_K}{\e}\int_\Omg \U(\rho^{(\e)}|  \rho)\,dx,
\]
where $C_0>0$ is independent of $\e > 0$.

This and Proposition \ref{prop_deri} yield
\begin{align*}
&\frac12\int_\Omg \rho^{(\e)} |u^{(\e)} -  u|^2\,dx + \frac{1}{\e}\lt(c_P - C_0c_K \rt)\int_\Omg \U(\rho^{(\e)}| \rho)\,dx+ \frac1{2\e} \int_0^t  \int_\Omg \rho^{(\e)} | u^{(\e)} -  u|^2\,dxd\tau\cr
&\quad \leq \frac12\int_\Omg \rho^{(\e)} |u^{(\e)} -  u|^2\,dx + \frac{c_P}\e\int_\Omg \U(\rho^{(\e)}| \rho)\,dx - \frac{c_K}{2\e} \int_\Omg (\rho^{(\e)} -  \rho)\Lmb^{\alpha - d}(\rho^{(\e)} -  \rho)\,dx \cr
&\qquad + \frac1{2\e} \int_0^t  \int_\Omg \rho^{(\e)} | u^{(\e)} -  u|^2\,dxd\tau\cr
&\quad \leq \frac12\int_\Omg \rho_0^{(\e)} |u_0^{(\e)} -  u_0|^2\,dx + \frac{c_P}\e\int_\Omg \U(\rho_0^{(\e)}| \rho_0)\,dx - \frac{c_K}{2\e} \int_\Omg (\rho_0^{(\e)} -  \rho_0)\Lmb^{\alpha - d}(\rho_0^{(\e)} -  \rho_0)\,dx\cr
&\qquad + \frac{C}{\e}\int_0^t  \int_\Omg \U(\rho^{(\e)} | \rho) \,dxd\tau   + C\e,
\end{align*}
where we chosen $c_K > 0$ small enough so that $c_P > C_0 c_K$. We then apply Gr\"onwall's lemma to have
\begin{align}\label{est_att}
\begin{aligned}
&\frac12\int_\Omg \rho^{(\e)} |u^{(\e)} -  u|^2\,dx + \frac{1}{\e}\lt(c_P - C_0c_K \rt)\int_\Omg \U(\rho^{(\e)}| \rho)\,dx+ \frac1{2\e} \int_0^t  \int_\Omg \rho^{(\e)} | u^{(\e)} -  u|^2\,dxd\tau\cr
&\quad \leq \frac12\int_\Omg \rho_0^{(\e)} |u_0^{(\e)} -  u_0|^2\,dx + \frac{c_P}\e\int_\Omg \U(\rho_0^{(\e)}| \rho_0)\,dx   + C\e.
\end{aligned}
\end{align}
On the other hand, 
\begin{align*}
\int_\Omg |\rho^{(\e)} u^{(\e)} -  \rho  u|\,dx &\leq \int_\Omg |\rho^{(\e)} -  \rho| |  u|\,dx + \int_\Omg \rho^{(\e)} | u^{(\e)} - u|\,dx \cr
&\leq C_1\|\rho^{(\e)} -  \rho\|_{L^\gamma} + \|\rho^{(\e)}\|_{L^1}^{1/2} \lt( \int_\Omg \rho^{(\e)} |u^{(\e)} -   u|^2\,dx\rt)^{1/2},
\end{align*}
where $C_1>0$ is explicitly given by
\[
C_1=
\left\{
\begin{aligned}
 |\T^d|^{\frac{1}{\gamma_*}}\|u\|_{L^\infty}  \quad &\mbox{when $\Omg = \T^d$}, \\[2mm]
 \| u\|_{L^{\gamma_*}} \quad &\mbox{when $\Omg = \R^d$},
\end{aligned}
\right.
\]
$\gamma_*$ is the H\"older's conjugate of $\gamma$, i.e. $\gamma_* = \gamma/(\gamma-1)$. This gives
\[
\int_0^t \|(\rho^{(\e)} u^{(\e)} -  \rho  u)(\tau,\cdot)\|_{L^1}^2\,d\tau \leq C\int_0^t \int_\Omg \U(\rho^{(\e)}| \rho)\,dxd\tau + C\int_0^t  \int_\Omg \rho^{(\e)} | u^{(\e)} -   u|^2\,dxd\tau.
\]
We finally use this, \eqref{est_att}, and Lemma \ref{lem_U2} to complete the proof.
\end{proof}

\begin{remark}\label{rmk_e2}  $L^\infty(\Omg)$-bound assumption on $u$ can be relaxed to $u \in L^\infty(0,T;L^{\gamma/(\gamma-1)}(\Omg))$, where $\gamma \geq 1$ is chosen as
\[
\gamma
\left\{
\begin{aligned}
&  > \frac{2d}{2d - \alpha}  \quad \mbox{when $\Omg = \T^d$}, \\
&  = \frac{2d}{2d - \alpha}  \quad \mbox{when $\Omg = \R^d$}.
\end{aligned}
\right.
\]
Under that assumption, we also have
\begin{align*}
&\sup_{0 \leq t \leq T}\|(\rho^{(\e)} -  \rho)(t,\cdot)\|_{L^\gamma}^2 + \int_0^T \|(\rho^{(\e)} u^{(\e)} -   \rho   u)(t,\cdot)\|_{L^1}^2\,dt\cr
&\quad \leq C\e\int_\Omg \rho_0^{(\e)} |u_0^{(\e)} -  u_0|^2\,dx + C\int_\Omg \U(\rho_0^{(\e)}| \rho_0)\,dx\cr
&\qquad  + C(1-\delta_{\gamma,1})\e^3\lt(\int_\Omg \rho_0^{(\e)} |u_0^{(\e)}|^2\,dx +\frac1\e\int_\Omg (\rho_0^{(\e)})^\gamma\,dx + \frac1\e \int_\Omg \rho_0^{(\e)} \Lmb^{\alpha-d}\rho_0^{(\e)}\,dx\rt)  + C\e^2,
\end{align*}
where $\delta_{\gamma,1}$ denotes the Kronecker delta function, i.e. $\delta_{\gamma,1} = 1$ if $\gamma=1$ and $\delta_{\gamma,1} = 0$ if $\gamma \neq 1$.
\end{remark}

%%%%%%%%%%%%%%%%%%%%%%%%%%%%%%%%%%%%%%%%%%%%%%%%
%
%
%
% \section*{Acknowledgement} 
%
%
%%%%%%%%%%%%%%%%%%%%%%%%%%%%%%%%%%%%%%%%%%%%%%%%

\section*{Acknowledgement} 
YPC has been supported by NRF grant (No. 2017R1C1B2012918) and Yonsei University Research Fund of 2019-22-021 and 2020-22-0505. IJJ has been supported  by a KIAS Individual Grant MG066202 at Korea Institute for Advanced Study, the Science Fellowship of POSCO TJ Park Foundation, and the National Research Foundation of Korea grant (No. 2019R1F1A1058486). 

% ----------------------------------------------------------------
\bibliographystyle{amsplain}
\bibliography{Euler_first_rel}
% ----------------------------------------------------------------

\end{document}